\documentclass[letterpaper, 10 pt, conference]{ieeeconf}  

\IEEEoverridecommandlockouts                              
\usepackage{amsmath}
\usepackage{mathtools,amssymb,amsfonts,subcaption}

\newtheorem{thm}{\bf Theorem}[section]

\newtheorem{deff}[thm]{\bf Definition}

\newcommand{\KL}{\mathcal{KL}}
\newcommand{\K}{\mathcal{K}}

\title{\LARGE \bf
Stability of Jordan Recurrent Neural Network Estimator
}

\author{Avneet Kaur*, Ruikun Zhou*, Jun Liu and Kirsten Morris
\thanks{This work was supported by the University of Waterloo and NSERC (Canada).}
\thanks{Avneet Kaur, Ruikun Zhou, Jun Liu and Kirsten Morris are with the Faculty of Mathematics, Department of Applied Mathematics, 
        University of Waterloo, 200 University Ave W, Waterloo, Canada 
        {\tt\small a93kaur@uwaterloo.ca, ruikun.zhou@uwaterloo.ca, j.liu@uwaterloo.ca, kmorris@uwaterloo.ca}. }%
}

\begin{document}

\maketitle
\thispagestyle{empty}
\pagestyle{empty}

\begin{abstract}

State estimation refers to determining the states of a dynamical system that starts from a noisy initial condition and evolves under process noise, based on noisy measurements and a known system model. For linear dynamical systems with white Gaussian noises of known mean and variance, Kalman filtering is a well-known method that leads to stable error dynamics for detectable systems. There are some non-optimal extensions to nonlinear systems. Recent work has used neural networks to develop estimators for nonlinear systems that optimize a criterion. Stability of the error dynamics is even more important than optimality. Jordan recurrent neural networks (JRNs) have a structure that mimics that of a dynamical system and are thus appealing for estimator design. We show that a JRN performs better than an extended Kalman filter(EKF) and unscented Kalman filter(UKF) for several examples. The main contribution of this paper is an input-to-state stability analysis of the error dynamics of JRNs. The stability of the error dynamics of several examples is shown. 
\end{abstract}

\section{INTRODUCTION}

State estimation refers to estimating the states of a dynamical system given a model and some measurements; both of which are subject to errors and disturbances.  It differs from system identification where there is no mathematical model. This is needed in some situations. However, often a model for the dynamics, although not perfect is available, and should be used in estimator design. The need for an estimator to run in real-time despite complex dynamics,  inaccurate measurements, incomplete information about the initial state, and errors in the model all mean that estimation, particularly of nonlinear systems, is challenging. 

Several approaches to state estimation exist dating from the early 1960s. Kalman filtering (KF) is a well known and widely used method for linear systems. If the disturbances on the model and the measurements are independent Gaussian noises, with known means and covariances then the Kalman filter yields an estimate of minimum error covariance. For nonlinear systems, several extensions of the Kalman filter are commonly used. In an extended Kalman filter (EKF), the system is linearised at each time-step and a Kalman filter for the linearized system is used \cite[e.g.]{dansimon}. It is widely used due to its simplicity and low computational overhead. Another extension is the unscented Kalman filter (UKF). This method relies on choosing sigma points to propagate state mean and covariance \cite[e.g.]{dansimon}. The UKF typically copes better with highly nonlinear systems, but is computationally expensive for high-order systems. 

Machine learning approaches are increasingly common. A deep long short-term memory network has been used to approximate the predicted and filtered state using  observations in \cite{gao_2019}. As for optimal control, a Hamilton-Jacobi-Bellman partial differential equation needs to be solved to find the minimum variance (or maximum likelihood) estimator. Its solution is very difficult for nonlinear systems since the number of independent variables equals the order of the system  (the curse of dimensionality). In \cite{kunisch}, a neural network is used to find solutions to the Hamilton-Jacobi-Bellman equation, assuming that a smooth solution exists.  In \cite{dl_based_observer}, a non-optimal design for discrete-time nonlinear systems is implemented using unsupervised learning. In another work, \cite{gencay1997nonlinear} compared the method of least squares, Elman recurrent neural network (ERN) and feedforward neural network (FNN) for noisy time series data with fully observable states and concluded that ERN provides the best estimation. Other papers that explore RNNs for state estimation include \cite{xie_deep_2024, jin2021new, wang2017state, chenna2004state}.
Recurrent neural networks have been shown to work well for system identification\cite{park2020analysis}. Jordan recurrent networks (JRNs) have been used for system identification by several researchers and have been shown to work as well as ERNs \cite{kasiran2012mobile, wu2019time, kuan_convergence_1994}. In \cite{KaurMorris} the structure of a Jordan recurrent network was extended to an LSTM and compared to an Elman long short-term memory network for estimator design. The training time considerably decreased when using JLSTM for state estimation. This is consistent with the fact that the structure of a Jordan RNN is closer to that of a dynamical system than that of a Elman RNN \cite{ian_goodfellow}. 

Just as stability is critical for a control system, in the absence of disturbances, the estimated state must converge to the true state. Small disturbances should increase the error by a small amount. In other words, the error dynamics must be input-to-state stable (ISS) \cite{Sontag1995,SontagWang1997}. It is known that for a linear system, if the measurements satisfy the weak assumption of detectability, the error dynamics are stable. It was recently shown \cite{AfsharGermMorris} that if a nonlinear system is locally detectable, then the error dynamics of an EKF are locally ISS. If the system is uniformly observable and can be put globally into the normal form, along with some other technical assumptions on the system dynamics then the EKF error dynamics are stable \cite[sec. 2.4]{GauthierKupka}.
For a discrete-time UKF satisfying certain assumptions, the estimation error remains bounded \cite{Xiong2006}. 
The stability of errors of RNNs as state estimators has not been well explored. However, there are several approaches to establishing the stability of RNNs by finding a  Lyapunov function \cite{liu2025physics, knight2008stability,zhou2022neural}. This suggested extending this work to show stability of error dynamics. The issue with ERNs in this context is that proving the stability of the estimator is difficult due to its dependence on the hidden neuron state at the previous time-step. This is not the case for a JRN. We therefore design an estimator using JRNs and utilize an input-to-state stability (ISS) approach to analyze the stability of the error dynamics. Finding Lyapunov functions is generally challenging. Several papers \cite{chang2019neural, zhou2022neural} have shown that Lyapunov functions can be attained and represented by expressions of the compositional structure of the neural network, the correctness being guaranteed by satisfiability modulo theories (SMT) solvers. We use this approach to establish a condition under which the error dynamics are stable if the original system is stable. Estimation of unstable systems is generally conducted in combination with a stabilizing feedback; our approach assumes that any necessary stabilization has been performed. 

The use of Jordan recurrent neural network for state estimation is described in Section \ref{sec:prelim}. Then in Section \ref{sec:stability} an ISS approach is used to provide a characterization of a Lyapunov function for the error dynamics. Data generation and training are described in Section \ref{sec:data_gen}. The approach is illustrated in Section \ref{sec:numerical_results} with three different discrete-time systems.


\begin{figure}[b]
    \centering
    \includegraphics[scale=0.25]{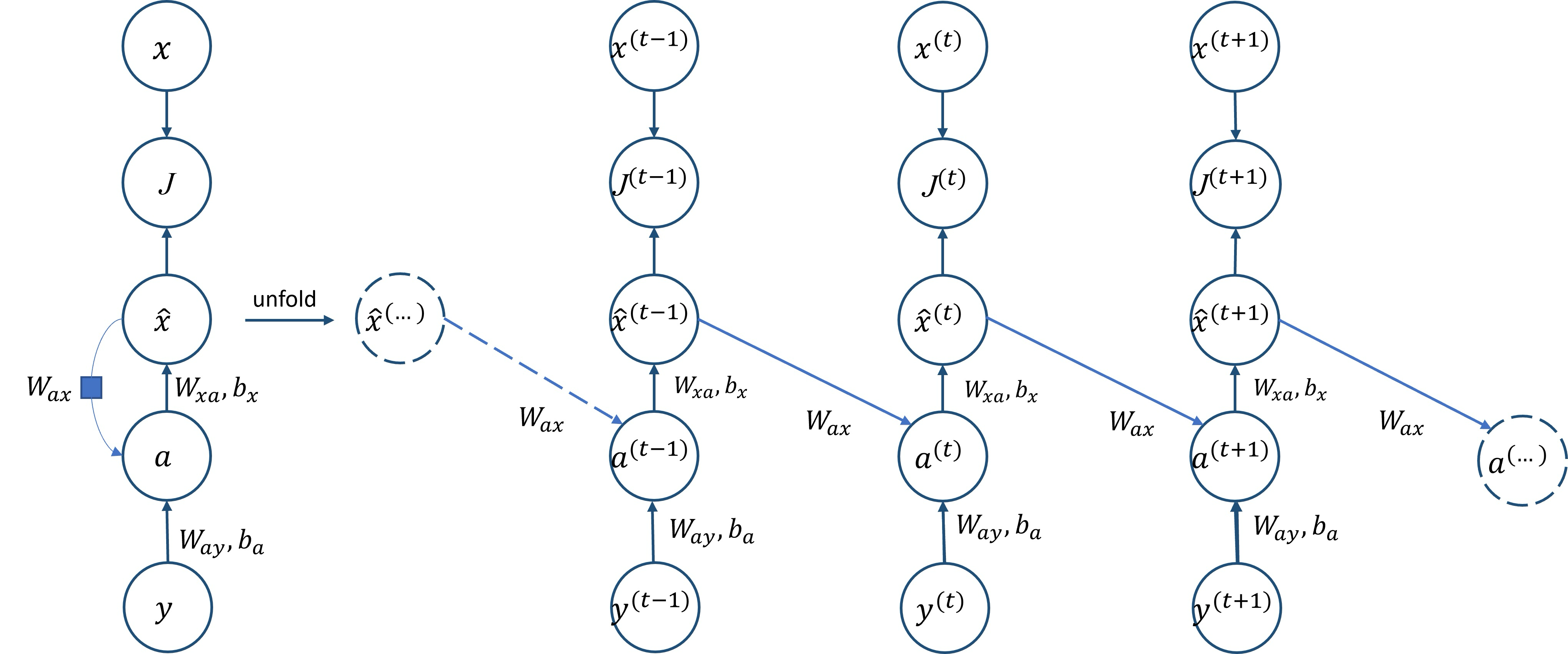}
    \caption{Jordan recurrent network (JRN) for state estimation showing output-to-hidden recurrent connections.}
    \label{fig:jordanRN}
\end{figure}
\section{Jordan recurrent neural network for state estimation}
\label{sec:prelim}
Consider a discrete-time nonlinear system 
\begin{align}
    \begin{split}
        x^{(t+1)}&=f(x^{(t)})+w^{(t+1)},  \\
        y^{(t+1)}&=h(x^{(t+1)})+v^{(t+1)}, \\
        x^{(0)}&=x^0+\bar{x}^0, 
        \label{eqn:noisy_disc_state_space_sys}
    \end{split}
\end{align}
where $x^{(t)} \in \mathcal{X} \subseteq \mathbb{R}^n$ is the state vector at time-step $t$, $w^{(t+1)} \in \mathbb{R}^n$ is the process noise vector at time-step $t+1$, $y^{(t+1)}\in \mathcal{Y} \subseteq \mathbb{R}^m$ is the measurement vector at time-step $t+1$, $v^{(t+1)}\in \mathbb{R}^m$ is the measurement noise vector at time-step $t+1$, $x^0 \in \mathbb{R}^n$ is the true initial condition and $\bar{x}^0 \in \mathbb{R}^n$ is the initial condition noise. Here, $\mathcal{X}$ and $\mathcal{Y}$ are the state space and measurement space, respectively, and $n \in \mathbb{N}$ is the total number of states and $m \in \mathbb{N}$ is the number of noisy measurements. The functions $f$ and $h$ are assumed to be sufficiently smooth.

Our aim is to estimate the state vector $x^{(t+1)}$ by $\hat{x}^{(t+1)}$ based on the measurement vector $y^{(t+1)}$ and the previous state vector estimate $\hat{x}^{(t)}.$ 

A JRN has recurrent connections from the output of the previous layer to the hidden state of the next layer. Thus, we use $y^{(t+1)}$  and $\hat{x}^{(t+1)}$ as inputs and outputs, respectively, of the JRN. The forward propagation of the network is
\begin{align}
    \begin{split}
        a^{(t+1)}&=\sigma(W_{ay} y^{(t+1)}+W_{ax} \hat{x}^{(t)}), \\
        \hat{x}^{(t+1)}&=W_{xa} a^{(t+1)},
        \label{eqn:proposed_jrn}
    \end{split}
\end{align}
where $\sigma$ is the activation function and $\hat{x}^{(t+1)}$ is the estimated state vector at time-step $t+1$. The hidden layer vector at time-step $t+1$ is given by $a^{(t+1)}.$ The weights and biases of the network are represented by $W_{ay}, W_{ax}$ and $W_{xa}$. The network is depicted in Fig. \ref{fig:jordanRN}. 

This is different from an ERN because the recurrent connections for an ERN are from the previous hidden layer $a^{(t)}$ to the next hidden layer $a^{(t+1)}$, whereas for a JRN, they are from the previous output $\hat{x}^{(t)}$ to next hidden layer $a^{(t+1)}.$ Thus, for an ERN, the term $W_{ax} \hat{x}^{(t)}$ would be replaced by $W_{aa} {a}^{(t)}$ where $W_{aa}$ is a weight matrix. Note that we are considering a bias-free network. 

\section{Input-to-state stability analysis}
\label{sec:stability}
To analyze the stability of the neural estimator, we first consider the case with no process and measurement noise in~\eqref{eqn:noisy_disc_state_space_sys}, i.e., $w^{(t+1)} = v^{(t+1)} = 0$ for all $t$, and $\bar{x}^0 = 0$: 
\begin{align}
    \begin{split}
        x^{(t+1)}&=f(x^{(t)}), \\
        y^{(t+1)}&=h(x^{(t+1)}).
        \label{eqn:disc_state_space_sys}
    \end{split}
\end{align}

The solution of system~\eqref{eqn:disc_state_space_sys} is denoted by $x^{(t)}(\xi)$ at time step $t$ with initial condition $x^{(0)} = \xi \in \mathcal{X}$. We define the error term as $e^{(t+1)} = x^{(t+1)}-\hat{x}^{(t+1)}$. Then, 
\begin{align}
    \begin{split}
        e^{(t+1)} &=  x^{(t+1)} - W_{xa} \sigma(W_{ay} y^{(t+1)}+W_{ax} \hat{x}^{(t)}) \\
       &= f(x^{(t)}) - W_{xa} \sigma(W_{ay} h(f(x^{(t)})) \\
       & \quad +W_{ax} x^{(t)} - W_{ax} e^{(t)}).
    \end{split}
\end{align}
Obviously, $e^{(t+1)}$ is a function of $e^{(t)}$ and $x^{(t)}$, and we call it the error system, denoted as follows,
\begin{equation}
    e^{(t+1)} \vcentcolon= g(e^{(t)}, x^{(t)}),
    \label{eqn:error_dynamics}
\end{equation}
where $e^{(t)} \in \mathcal{E} \subseteq \mathbb{R}^n$, and $\mathcal{E}$ is the space for the error term. In a similar manner, we denote the solution to system~\eqref{eqn:error_dynamics} as $e^{(t)}(\eta, x)$ with the initial condition  $e^{(0)} = \eta \in \mathcal{E}$.

We introduce  Lyapunov functions for the error dynamics, regarding $x^{(t)} $ as the input in \eqref{eqn:error_dynamics}.
\begin{deff}
    A continuous function $\alpha: \mathbb{R}_{\geq0} \rightarrow \mathbb{R}_{\geq0}$ is a $\K$-function if it is strictly increasing and $\alpha(0) = 0$. It is a $\K_{\infty}$-function if it is a $\K$-function and $\alpha(r
    ) \rightarrow \infty$ as $r \rightarrow \infty.$
\end{deff}
\begin{deff}
    A continuous function $\beta: \mathbb{R}_{\geq0} \times \mathbb{R}_{\geq0} \rightarrow \mathbb{R}_{\geq0}$ is a $\KL$-function if, for each $t$, $\beta(\cdot, t)$ is a $\K$-function with respect to $r$, and for each $r$, $\beta(r,\cdot)$ is decreasing with respect to $t$, and $\beta(r,t) \rightarrow 0$ as $t \rightarrow \infty.$
\end{deff}

\begin{deff}
    The origin is globally asymptotically stable for \eqref{eqn:disc_state_space_sys} if, for each $\varepsilon > 0$, there is a $\delta = \delta(\varepsilon) > 0$ such that 
    \begin{equation}
        \|\xi \| < \delta \implies \|x^{(t)} \| < \varepsilon, \forall t \geq 0, \text{and} \lim_{t \rightarrow \infty}  x^{(t)} = 0.
    \end{equation}
    In addition, there exists a $\KL$-function $\beta$ such that 
    \begin{equation}
        \| x^{(t)}(\xi) \| \leq \beta(\|\xi\|, t) \quad \forall \xi \in \mathcal{X}, \forall t \in \mathbf{Z}_+.
    \end{equation}
\end{deff}

\begin{deff} \cite{jiang2001input}
The discrete-time system~\eqref{eqn:error_dynamics} is input-to-state stable (ISS) if there exists a $\KL$-function $\beta$ and a $\K$-function $\gamma$, such that for each $t \in \mathbf{Z}_+$, 
\begin{equation}
    |e^{(t)}(\eta, x)| \leq \beta( \|\eta \|, t) + \gamma(\|x\|), 
\end{equation}
for all $\eta \in \mathcal{E}$ and for all $x \in \mathcal{X}$.
\end{deff}

\begin{deff} \cite{jiang2001input}
    A continuous function $V: \mathbb{R}^n \rightarrow \mathbb{R}_{\geq0}$ is called an ISS-Lyapunov function for the discrete-time system~\eqref{eqn:error_dynamics} if there exist $\K_{\infty}$-functions $\alpha_1, \alpha_2, \alpha_3$, and a $\K$-function $\gamma$ such that
    \begin{equation}
        \alpha_1(\|e^{(t)}\|) \leq V(e^{(t)}) \leq \alpha_2(\|e^{(t)}\|),\quad \forall e^{(t)} \in \mathcal{E},
        \label{eqn:condition_positive}
    \end{equation}
    and
    \begin{align}
        V(g(e^{(t)}, x^{(t)})) - V(e^{(t)}) & \leq -\alpha_3(\|e^{(t)}\|) + \gamma(\|x^{(t)}\|) \notag\\
        & \forall e^{(t)} \in \mathcal{E}, \forall x^{(t)} \in \mathcal{X}.    \label{eqn:condition_negative}
    \end{align}
    \label{Def:ISS}
\end{deff}

\begin{thm}
    If the error system~\eqref{eqn:error_dynamics} with $x^{(t)}$ as input is ISS and the origin of the discrete-time system~\eqref{eqn:disc_state_space_sys} is globally asymptotically stable, then the origin of the cascade system~\eqref{eqn:disc_state_space_sys} and~\eqref{eqn:error_dynamics} is globally asymptotically stable.
    \label{thm:stable}
\end{thm}
\begin{proof}
The solutions of~\eqref{eqn:disc_state_space_sys} and~\eqref{eqn:error_dynamics} satisfy:
\begin{align*}
    \| x^{(t)}(\xi) \| & \leq \beta_1(\|\xi\|, t),\\
  \| e^{(t)}(\eta, x)\| & \leq \beta_2(\|\eta\|, t) + \gamma(\|x^{(t)}\|), 
\end{align*}
where $\beta_1, \beta_2$ are $\KL$-functions. Thus,
\begin{equation}
    \| e^{(t)}(\eta, x)\| \leq \beta_2(\|\eta\|, t) + \gamma(\beta_1(\|\xi\|, t)). 
\end{equation}
Let $s^{(t)}$ denote the concatenation of the state $x^{(t)}$ and the error state $e^{(t)}$, and $\zeta$ denote the origin of this cascade system. We have $\|x^{(t)} \| \leq \| s^{(t)} \|$ and $\|e^{(t)} \| \leq \| s^{(t)} \|$, and $ \| s^{(t)} \| \leq \| x^{(t)} \| + \| e^{(t)} \|$ (because for any 2 non-negative real numbers $a$ and $b, \sqrt{a^2+b^2}\leq a+b$). Defining $\beta(\cdot, \cdot) = \beta_1(\cdot, \cdot) + \beta_2(\cdot, \cdot) + \gamma(\beta_1(\cdot, \cdot)) $, this yields
 \begin{equation}
        \| s^{(t)}(\zeta) \| \leq \beta(\|\zeta\|, t) \quad \forall t \in \mathbf{Z}_+.
    \end{equation}
It can be easily verified that $\beta$ is a $\KL$-function.
\end{proof}

Two different approaches were used for finding the Lyapunov function, depending on whether the system is linear or nonlinear.

{\bf Linear Systems.}  Consider a linear discrete-time system defined by, for matrices $A$ and $H$,
\begin{align}
    \begin{split}
        x^{(t+1)}&=A x^{(t)},  \\
        y^{(t+1)}&=H x^{(t+1)}. 
        \label{eqn:linear_disc_state_space_sys}
    \end{split} 
\end{align}

By using the identity function as the activation function $\sigma$ in \eqref{eqn:proposed_jrn}, the error dynamics of the linear state estimator can be written as
\begin{align}
    \begin{split}
        e^{(t+1)} &=Ax^{(t)}-W_{xa}(W_{ay} y^{(t+1)}+W_{ax} \hat{x}^{(t)})  \\
        &=Ax^{(t)}-W_{xa}(W_{ay}Hx^{(t+1)} + W_{ax} (x^{(t)} - e^{(t)}))\\
        &=(A-W_{xa}W_{ay}HA -W_{xa}W_{ax}) x^{(t)} + W_{xa}W_{ax}e^{(t)}.
    \end{split}
\end{align}
Defining $\mathcal{A}  = W_{xa}W_{ax}$ and $\mathcal{B} = A-W_{xa}W_{ay}HA -W_{xa}W_{ax}$,  the error system is a linear system with $x^{(t)}$ as the input:
\begin{equation}
    e^{(t+1)} = \mathcal{A} e^{(t)} + \mathcal{B} x^{(t)} .
\end{equation}
If this discrete-time system is ISS then there is a quadratic Lyapunov function
\begin{equation}
    V(e) = e^T P e,
\end{equation}
where $P$ is a positive definite matrix obtained by solving 
\begin{equation}
    \mathcal{A}^T P \mathcal{A} - \mathcal{A} + Q = 0.
\end{equation}
Here, $Q$ is a symmetric positive definite matrix. In this case, it is easy to show that both properties in Def.~\ref{Def:ISS} are satisfied with $\alpha_1 (x) = \lambda_{min}(P)x^2$, $\alpha_2 (x) = \lambda_{max}(P)x^2$, $\alpha_3 (x) =\frac{1}{2} \lambda_{min}(Q)x^2$, and $\gamma(x) = (\frac{2|\mathcal{A}^T P \mathcal{B}|^2}{\lambda_{min}(Q)} + |\mathcal{B}^T P \mathcal{B}|^2)x^2$ ~\cite{jiang2001input}.

{\bf Nonlinear systems. }
For nonlinear systems,  a counterexample-guided method with verification provided by SMT solvers was used to synthesize the ISS-Lyapunov functions for the error system~\eqref{eqn:error_dynamics}. 

Inspired by \cite{abate2021fossil}, we use a one-hidden layer feed-forward neural network with zero bias terms for all layers to learn a Lyapunov function of the following form
\begin{equation}
    V(e;\theta) = \sigma(W_2 \sigma(W_1 e )),
\end{equation}
where $W_1$ and $W_2$ are the weights for the hidden layer and output layer respectively, $\theta$ denotes all the hyper-parameters, $\sigma$ is the activation function, $e$ is the input vector, which is the shorthand notation for the error state at time $t$. Similarly, $x$ denotes the state at time $t$. 

We use the square function as the activation function $\sigma$ in this network, which results in sum-of-squares (SOS)-like quadratic Lyapunov functions. Then $V(0) = 0$. As a result, a valid ISS Lyapunov function is attained when properties~\eqref{eqn:condition_positive} and ~\eqref{eqn:condition_negative} are satisfied. Its falsification constraints can then be written as a first-order logic formula over the real numbers. This yields 
\begin{align}
    \begin{split}
        & \left(\sum_{i=1}^{n} e_{i}^{2} \leq r_e  \vee \sum_{i=1}^{n} x_{i}^{2} \leq r_x \right) \wedge \\
& \bigg( (V(x) - \alpha_1(|e|) \leq 0) \vee (V(x) - \alpha_2(|e|) \geq 0) \vee \\
& (V(g(e, x)) - V(e) + \alpha_3(|e|) - \gamma(|x|) \geq 0)
\bigg),
\label{SMT}
    \end{split}
\end{align}
where, $r_e$ and $r_x$ are the radii of the so-called valid region for the error state and the state respectively, on which we verify the condition using SMT solvers. When the SMT solver returns \verb|UNSAT|, a valid ISS Lyapunov function is obtained. Otherwise, it produces counterexamples that can be added to the training dataset of the neural network for finding Lyapunov function candidates. Due to the nature of the SMT solvers, we typically need to exclude a small region around the origin when verifying the falsification constraints. Around the origin, the linearization of the nonlinear system dominates, and we  implement the method for the linearized model,  described above,  to provide stability guarantees. We refer the readers to~\cite{liu2025physics} for a detailed  proof and algorithm.

The loss function is consistent with the falsification constraints: 
\begin{align}
 L(\theta)&= \frac{1}{MN} \sum_{j=1}^{M} \sum_{i=1}^{N} \max \bigg( 0, V_{\theta}(g(e_i, x_j)) - V_{\theta}(e_i) \notag\\
 & + \alpha_3(|e_i|) - \gamma(|x_j|) \bigg) + \max \bigg( 0, V_{\theta}(e_i) - \alpha_2(|e_i|) \bigg) \notag\\
 & + \max \bigg( 0,  \alpha_1(|e_i|) -V_{\theta}(e_i) \bigg).
     \label{eqn:cost}
\end{align} 
This is known as the positive penalty for the violation of conditions in Definition~\ref{Def:ISS}.


\section{Implementation}
\label{sec:data_gen}
The datasets are obtained by discretizing common ordinary differential equations. We use a zero order hold discretization for linear systems. For nonlinear continuous-time state space systems, we use the RK-45 discretization using Python's $scipy.integrate.solve\_ivp$ function. The initial condition is considered to be Gaussian with mean sampled uniformly from the interval $[-1,1] \times [-1,1] \subset \mathbb{R} \times \mathbb{R}$ for each sequence and known covariance. Process and measurement noises are assumed to be Gaussian with zero mean and known covariance. For the systems under consideration, the above 3 covariance matrices are assumed to be $0.01 \times I$ where $I$ is an identity matrix of appropriate dimensions. Also, knowledge of the functions $f$ and $h$ in \eqref{eqn:noisy_disc_state_space_sys} is used to generate data. 

For the linear system, we consider a total of $100$ sequences and $200$ for non-linear systems. The data set is divided into three parts: training, validation and testing in the ratio $80:10:10$ sequences. Each of the generated sequences is independent of the other sequences in the dataset. 

A custom network with forward propagation as in \eqref{eqn:proposed_jrn} is implemented. The weight matrix $W_{ay}$ is initialized using Xavier uniform distribution while weight matrices $W_{ya}$ and $W_{xa}$ are initialized to be (semi) orthogonal matrices using \textit{Pytorch}'s $torch.nn.init$ module. The function $\sigma(z)$ is chosen to be equal to $z$ for the linear system and $\tanh(z)$ for the nonlinear systems. The backward propagation is implemented using \textit{PyTorch}'s $backward()$ function. We consider the mean squared error loss function at time-step $t+1$ for each sequence 
\begin{equation}
    J_{i}^{(t+1)}(\phi)=\frac{1}{n}\sum_{i=1}^{n}{(x^{(t+1)}_{i}-\hat{x}^{(t+1)}_{i,\phi})}^{2}, 
\end{equation}
where $1\leq i\leq n$ and $n$ is the number of states, $x^{(t+1)}_{i}$ represents the true value of state $x_i$ at time-step $t+1$ and $\hat{x}^{(t+1)}_{i,\phi}$ represents the estimated value of state $x_i$ at time-step $t+1$ by the JRN where $\phi$ denotes the weights of the network. 

Hyperparameter tuning was performed for all examples individually. \textit{Adam} optimization is used for training the network. The optimal learning rate is chosen based on the network's performance on validation dataset over a range of learning rates varying from $10^{-1}$ to $10^{-4}$. A batch size of $40$ and a hidden unit size of $50$ are considered for each example in Section~\ref{sec:numerical_results}. Early stopping with a fixed patience value to decide the number of epochs needed is used. In each epoch, network's performance on validation data is used as the measure to decide whether to proceed or keep training for another epoch. Maximum number of epochs is set to $600$ in case early stopping is not reached. For each lower validation loss value than the previous, the model is saved and after training reloaded to the one corresponding to the least validation loss. 

We test our estimators for 3 different systems in Section~\ref{sec:numerical_results}. We compare the results with KF for the linear system as it is an optimal filter for systems with additive white Gaussian noise. For nonlinear systems with additive white Gaussian noise, the best filters are considered to be EKF and UKF. Hence, we compare our nonlinear systems results with them. Standard implementations of these 3 filters are used, see for example \cite{dansimon}. For EKF, the equilibrium point around which the systems are linearized is considered to be the origin. 

We  compare JRN, EKF and UKF graphically using average error at time-step $t$ over all features and all test sequences
\begin{equation}
    \text{Error}(t)=\frac{1}{m_{test} n} \sum_{k=1}^{m_{test}} \sum_{i=1}^{n} (x_i^{(t)[k]}-\hat{x}_i^{(t)[k]})^2,
\end{equation}
where $m_{test}$ is the number of test sequences, $n$ is the number of states, $x_i^{(t)[k]}$ represents the true value of state $x_{i}$ at time-step $t$ for the $k^{th}$ sequence and $\hat{x}_i^{(t)[k]}$ represents the estimated value of state $x_{i}$ at time-step $t$ for the $k^{th}$ sequence.
We also use root mean square error (RMSE) to compare different methods. It is calculated as
\begin{equation}
    \text{RMSE}=\sqrt{\frac{1}{m_{test} n T} \sum_{k=1}^{m_{test}} \sum_{j=1}^{T} \sum_{i=1}^{n} (x_i^{(j)[k]}-\hat{x}_i^{(j)[k]})^2},
\end{equation}
where $m_{test}$ is the number of test sequences, $T$ is the number of time-steps in each sequence, $n$ is the number of states, $x_{i}^{(j)[k]}$ represents the true value of state $x_{i}$ at time-step $j$ for the $k^{th}$ sequence and $\hat{x}_{i}^{(j)[k]}$ represents the estimated value of state $x_{i}$ at time-step $j$ for the $k^{th}$ sequence.

\begin{figure*}
    \centering
    \begin{subfigure}{0.30\textwidth}
        \includegraphics[width=\linewidth]{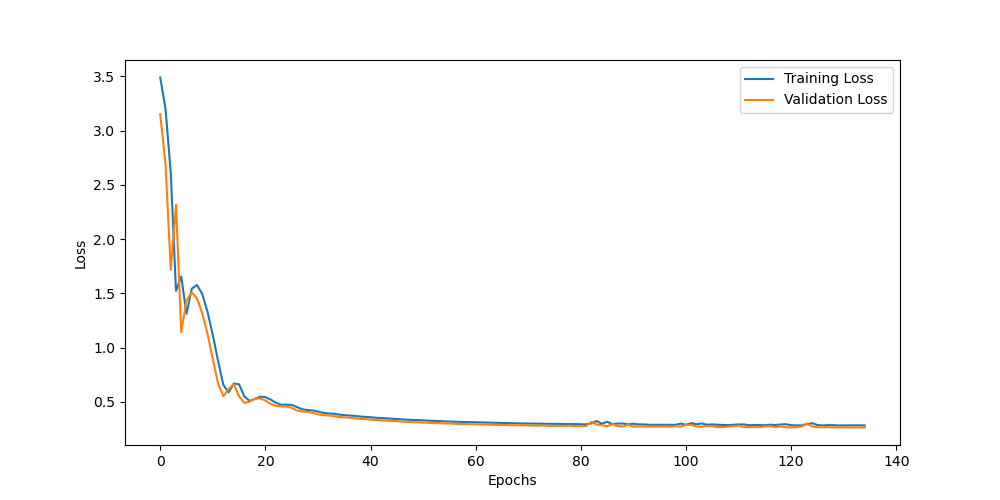}
        \caption{}
        \label{fig:loss_ms}
    \end{subfigure}
    \hfill
    \begin{subfigure}{0.30\textwidth}
        \includegraphics[width=\linewidth]{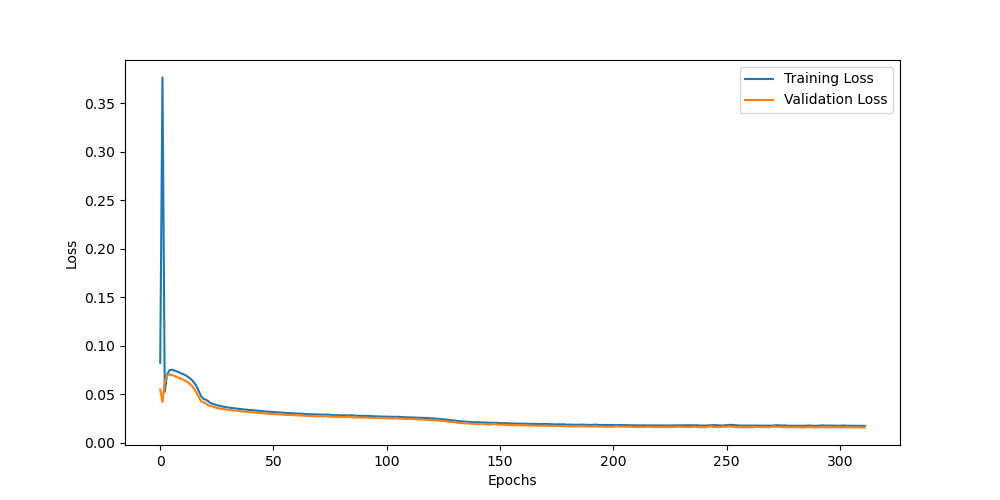}
        \caption{}
        \label{fig:loss_dp}
    \end{subfigure}
    \hfill
    \begin{subfigure}{0.30\textwidth}
        \includegraphics[width=\linewidth]{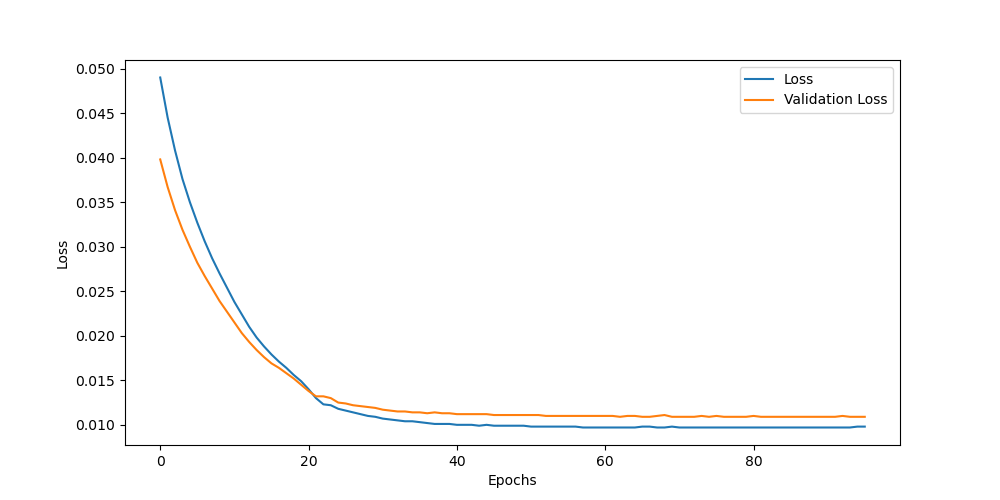}
        \caption{}
        \label{fig:loss_rvp}
    \end{subfigure}
    \hfill
    \begin{subfigure}{0.30\textwidth}
        \includegraphics[width=\linewidth]{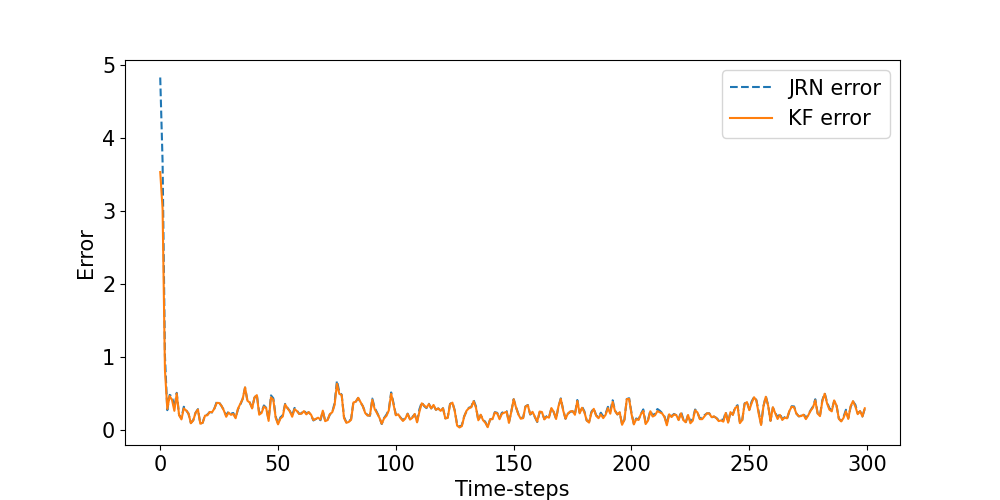}
    \caption{}
    \label{fig:errors_ms}
    \end{subfigure}
    \hfill
    \begin{subfigure}{0.30\textwidth}
        \includegraphics[width=\linewidth]{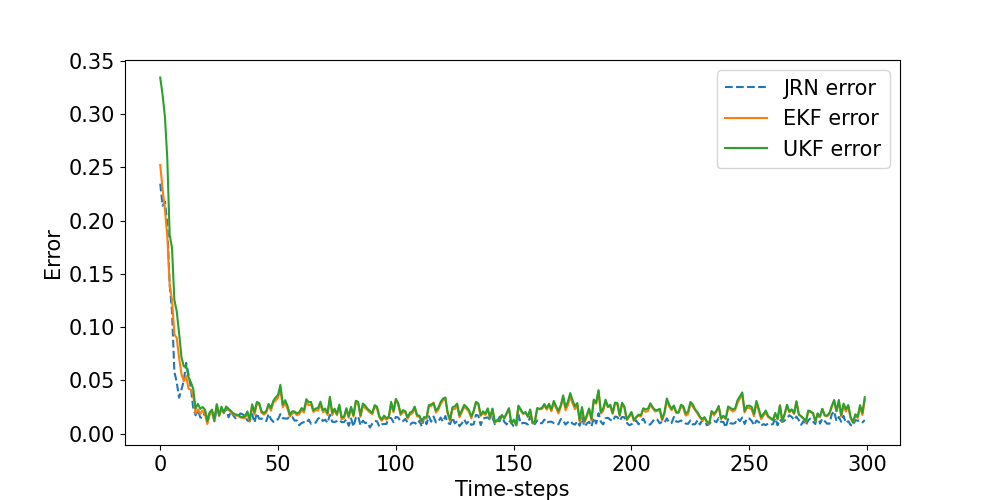}
        \caption{}
        \label{fig:errors_dp}
    \end{subfigure}
    \hfill
    \begin{subfigure}{0.30\textwidth}
        \includegraphics[width=\linewidth]{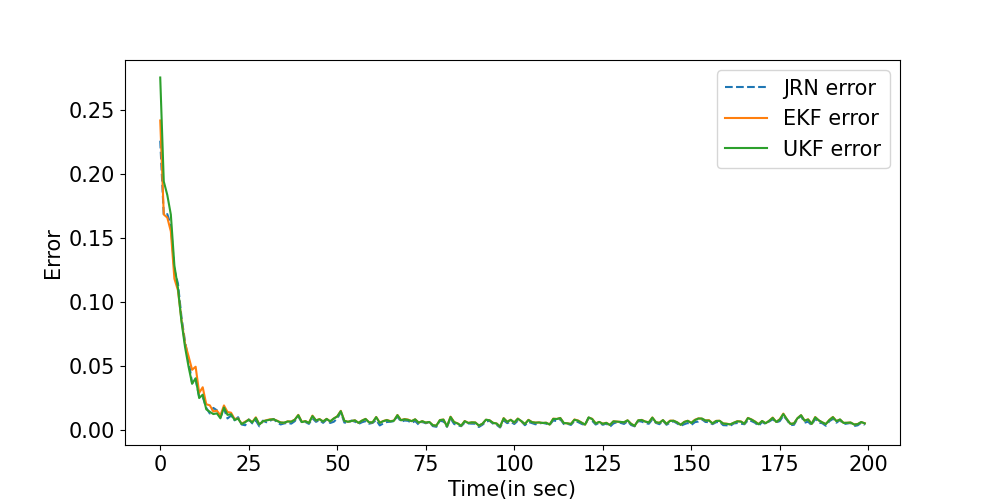}
        \caption{}
        \label{fig:errors_rvp}
    \end{subfigure}
    \caption{Figures \ref{fig:loss_ms}, \ref{fig:loss_dp} and \ref{fig:loss_rvp} show the  training and validation values of JRNs for the mass spring damper system, down pendulum and reversed Van der Pol oscillator respectively.
    Figures  \ref{fig:errors_ms} \ref{fig:errors_dp} and \ref{fig:errors_rvp} show average errors at each time-step for JRN, EKF and UKF state estimates for mass spring damper system, down pendulum and reversed Van der Pol oscillator respectively.}
    \label{fig:nonlinear_loss_comparison}
\end{figure*}

 \begin{figure}
    \centering
    \includegraphics[scale=0.65]{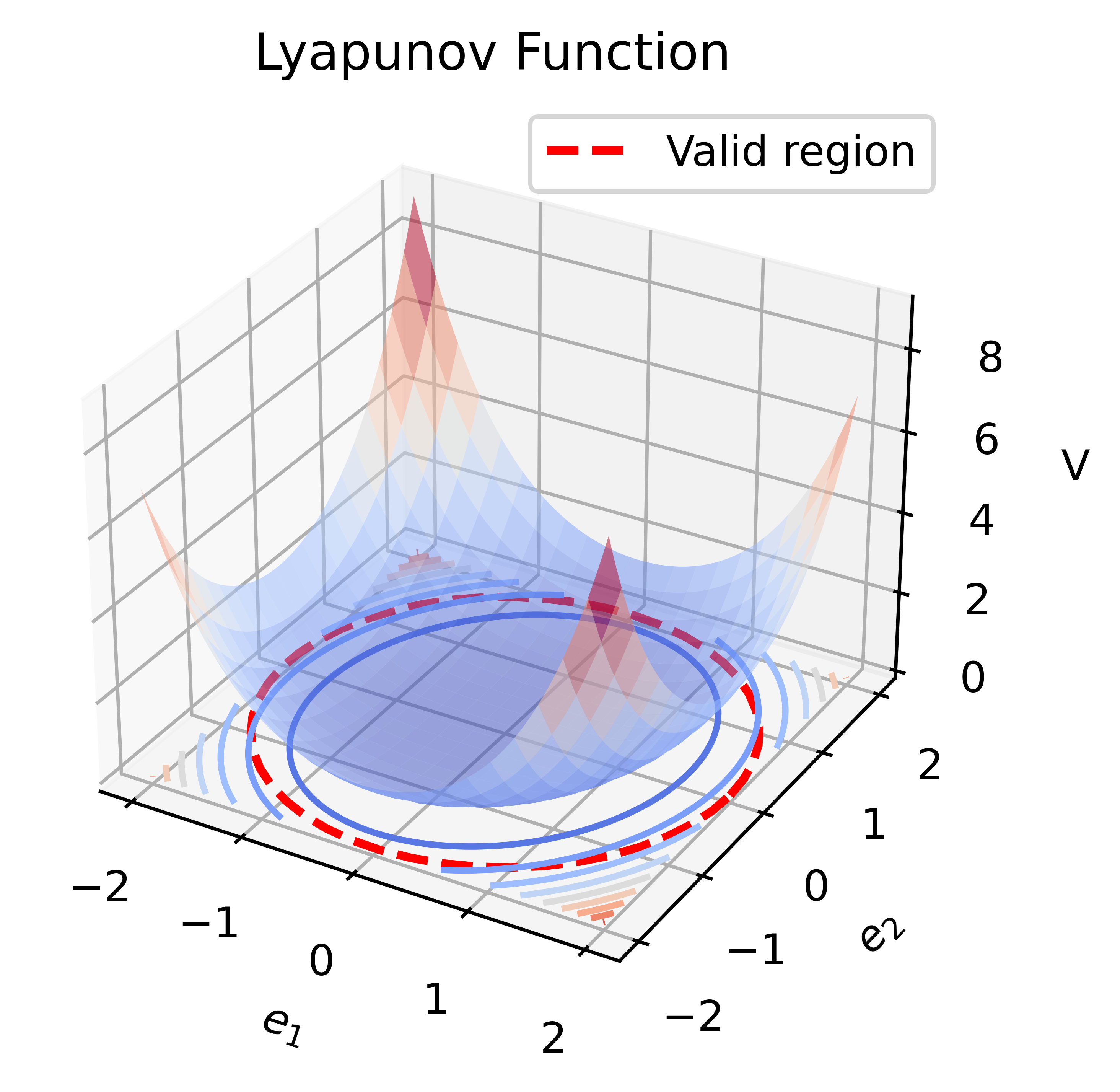}
    \caption{Learned ISS Lyapunov function for a down pendulum.}
    \label{fig:lf-pendulum}
\end{figure}
\begin{table}[b]
\caption{Training hyperparameters}
\begin{center}
    \begin{tabular}{|c|c|c|c|}
        \hline
        \textbf{State space system} &\textbf{Time-steps} & \textbf{Learning rate} & \textbf{Patience} \\
        \hline
        Mass spring & 300 & 0.01 & 10\\
        Down pendulum & 300 & 0.01 & 25\\
        r. Van der Pol & 200 & 0.001 & 25\\
        \hline
    \end{tabular}
    \label{tab:optimization}
    \end{center}
\end{table}

\begin{table}[b]
\caption{Parameters}
\begin{center}
    \begin{tabular}{|c|c|}
        \hline
        \textbf{State space system} &\textbf{Parameters} \\
        \hline
        Mass spring & $m=10$ kg, $b=6$ kg/sec, $k=800$ kg/sec$^2$ \\
        Down pendulum & $m=2$ kg, $b=0.9$ kg/sec, $l=1 m.$\\
        \hline
    \end{tabular}
    \label{tab:parameters}
    \end{center}
\end{table}
\section{Examples}
\label{sec:numerical_results}
In this section, we illustrate the proposed method with 3 discrete-time systems obtained as explained in \ref{sec:data_gen}. The ODEs considered are mass-spring damper system, down pendulum and reversed Van der Pol oscillator. The measurement vector at time-step $t$ for each of the systems is $y^{(t)}=x_1^{(t)}+\nu^{(t)}$ where $x_1(t)$ refers to the position and $\nu^{(t)}$ is the measurement noise at time-step $t$. Important parameters for each of these systems are summarised in Tables \ref{tab:optimization} and \ref{tab:parameters}.A remote Ubuntu server was used for all calculations. The codes used are available on Github at https://github.com/avneetkaur96/JRNs-for-state-estimation.
\begin{table}[b]
\caption{Root mean square error (RMSE) for E(KF), UKF and JRN}
\begin{center}
    \begin{tabular}{|c|c|c|c|c|c|}
        \hline
        \textbf{State space} &\multicolumn{3}{|c|}{\textbf{Estimator}}\\
        \cline{2-4}
        \textbf{system} & \textbf{\textit{E(KF)}} & \textbf{\textit{UKF}} & \textbf{\textit{JRN}}\\
        \hline
        Mass spring & \textbf{0.5132} & - & 0.5268\\
        Down pendulum & 0.1594 & 0.1693 & \textbf{0.1298} \\
        r. Van der Pol & 0.1153 & 0.1158 & \textbf{0.1102} \\
        \hline
    \end{tabular}
    \label{tab:MSE}
    \end{center}
\end{table}

\begin{table}[b]
\caption{Best epoch and training time for JRN}
\begin{center}
    \begin{tabular}{|c|c|c|}
        \hline
        \textbf{State space system} &\textbf{Best Epoch} & \textbf{Training time(in sec)} \\
        \hline
        Mass spring & 135 & 20\\
        Down pendulum & 313 & 101\\
        r. Van der Pol & 96 & 32\\
        \hline
    \end{tabular}
    \label{tab:Epochntraintime}
    \end{center}
\end{table}

The performance of the JRN and the KF for mass-spring system state estimation is summarized in Figures \ref{fig:errors_ms}  and \ref{fig:loss_ms}. As shown in Figure \ref{fig:errors_ms}, both the JRN and KF achieve rapid error reduction within the first few time steps, with the KF maintaining slightly lower error overall as expected due to its optimal nature. The JRN closely tracks the KF’s performance after the initial steps, indicating effective learning of the system dynamics. The difference of estimation in the first few timesteps can be attributed to the fact that the KF is provided with an estimate of the initial condition while the JRN has no information of the initial condition. The training and validation loss curves for the JRN (Figure \ref{fig:loss_ms}) demonstrate fast convergence, with loss stabilizing after approximately 40 epochs. The close alignment of training and validation loss throughout training indicates good generalization and minimal overfitting. The RMSE on the test set is $0.5132$ for the KF and $0.5268$ for the JRN (Table \ref{tab:MSE}, confirming that the JRN provides comparable estimation accuracy. The model was restored to its state at the best epoch for testing. For this system, the origin is asymptotically stable. We then implement the method  for stability of linear systems. Setting  $Q = I$, it  is easy to establish that $P = \begin{bmatrix}
45.504 & -3.5737 \\
-3.5737 & 1.4461
\end{bmatrix}.$ With this ISS-Lyapunov function, by Theorem~\ref{thm:stable}, both the original system and the error system are asymptotically stable. This proves the stability of the NN  estimator error.

As shown in Figures \ref{fig:errors_dp} and \ref{fig:errors_rvp}, all three approaches exhibit a rapid reduction in error within the first 25 time-steps, stabilizing thereafter.JRN achieves the lowest estimation error over time, outperforming both EKF and UKF for the down pendulum. The errors for both systems and for all methods stabilize after approximately $50$ time-steps, with JRN maintaining a consistently lower error throughout for the down pendulum. Figures \ref{fig:loss_dp} and \ref{fig:loss_rvp} demonstrate the convergence of the JRN training process. Both training and validation losses decrease steadily, indicating good generalization and absence of overfitting. The JRN achieves the lowest RMSE compared to the EKF and UKF for both the nonlinear systems, indicating improved accuracy (See Table \ref{tab:MSE}). These results highlight the effectiveness and efficiency of the JRN in nonlinear state estimation tasks. For both nonlinear systems, an ISS Lyapunov function was learned using a neural network and verified by an SMT solver. For the down pendulum, this function is shown in Fig.~\ref{fig:lf-pendulum}, verified by an SMT solver, dReal~\cite{gao2013dreal}, on a $[-2,2] \times [-2,2]$ for the error states. In this case, we used $\alpha_1(\cdot) = \alpha_3(\cdot) = 0.01|\cdot|$ and $\alpha_2(\cdot) = \gamma(\cdot) = 100|\cdot|$. This provides the asymptotic stability for the error dynamics by Theorem~\ref{thm:stable}.

    


\section{Conclusions and future work}
The Jordan recurrent network based estimator achieved results as good as KF for the linear system and  better than EKF and UKF for the chosen nonlinear systems. While the training time is a drawback in case of neural networks, the testing time of JRN-based estimators is much lower than the classical approaches of UKF and EKF. Current work focuses on extending the Jordan architecture to a long short-term neural network \cite{KaurMorris}. 
 
The main contribution of this paper is to  verify the stability of the error dynamics of the RNN-based estimator using an ISS approach, by taking advantage of the structural properties of an unbiased JRN and recent results in learning Lyapunov functions. For linear systems, quadratic Lyapunov equations can be computed analytically for the error dynamics to provide stability. For nonlinear systems, a counter-example guided method was used to learn an ISS Lyapunov function. The linear approach is used near the equilibrium. The effectiveness of the  method has been illustrated with three examples. 

Future work will extend this approach to higher-order systems. Addressing the issue of estimation of unstable systems is also of interest. Another potential future research direction is to consider stability for the case where biases are non-zero, and for networks with long short term memory.


\bibliographystyle{ieeetr}        
\bibliography{root.bib}
\end{document}